\documentclass{article}
\usepackage{mathtools}
\usepackage{amssymb}
\usepackage{tikz-cd}

\usepackage[all]{xy}
\usepackage{graphicx}
\usepackage{amsmath}

\usepackage{scalerel}

\usepackage{lipsum}
\usepackage{ntheorem}
\newtheorem{theorem}{Theorem}
\newtheorem*{theorem-non}{Theorem}

\newtheorem{definition}[theorem]{Definition}

\newtheorem{proposition}[theorem]{Proposition}

\newenvironment{proof}[1][Proof]{\textbf{#1.} }{\ \rule{0.5em}{0.5em}}

\addtocounter{footnote}{-1}

\begin{document}
       
\begin{center}
{A note on 
$\mathcal{D}$%
-modules on the projective spaces of a class of
$G$%
-representations}

{\let\thefootnote\relax}\\

 \bigskip
{\large Philibert Nang}
\footnote{D\'epartement de Math\'ematiques, \'Ecole Normale Sup\'erieure, Libreville, Gabon. E-mail: philnang@gmail.com}
\footnote{Max-Planck Institute for Mathematics (MPIM), Vivatsgasse 7, Bonn 53111 Germany. E-mail: pnang@mpim-bonn.mpg.de}

\end{center}

\bigskip
     
\noindent

\textbf{Abstract:}
\\

Consider
$(G, V)$
a finite-dimensional representation of a connected reductive complex Lie group 
$G$
and 
$\mathbb{%
P}\left( V\right) $ 
the projective space of
$V$%
.
Denote by 
$G'$
the derived subgroup of
$G$
and assume that the categorical quotient is one dimensional.  In the case where the representation
$(G, V)$
is also multiplicity-free, it is known from Howe-Umeda \cite{H-U} that 
the algebra of
$G$%
-invariant differential operators
$\Gamma\left(V, \mathcal{D}_V\right)^G$
is a commutative polynomial ring. Suppose that the representation
$(G, V)$
satisfies the abstract Capelli condition: 
$(G, V)$
is an irreducible multiplicity-free representation such that the Weyl algebra
$\Gamma\left(V, \mathcal{D}_V\right)^G$
is equal to the image of the center of the universal enveloping algebra of
$\mathrm{Lie}(G)$
under the differential 
$\tau: \mathrm{Lie}(G) \longrightarrow \Gamma\left(V, \mathcal{D}_V\right)$
of the 
$G$%
-action.\\
Let
$\mathcal{A}$
be the quotient algebra of all 
$G'$%
-invariant differential operators by
those vanishing on 
$G'$%
-invariant polynomials. The main aim of this paper is to prove that there is an
equivalence of categories between the category of regular holonomic 
$%
\mathcal{D}_{\mathbb{P}\left( V\right) }$%
-modules on the complex projective space 
$\mathbb{P}\left( V\right) $
and the quotient category of finitely generated graded 
$\mathcal{A}$%
-modules modulo those supported by 
$\left\{ 0\right\} $%
. This result is a generalization of
\cite[Theorem 3.4]{N0} and of \cite[Theorem 8]{N00}. As an application we give an algebraic/combinatorial
classification of regular holonomic 
$\mathcal{D}_{\mathbb{P}%
\left( V\right) }$%
-modules on the projective space of skew-symmetric matrices.
\\

\\


\section{Introduction}

Let
$V$ 
be a finite dimensional complex vector space, and 
$%
\mathbb{P}\left(V\right)$
its associated projective space with the actions by a connected reductive Lie group
$G$
and
$\mathbb{P}\left(G\right)$
respectively. Denote by 
$G' = [G, G]$
the derived subgroup of
$G$%
.
Assume that the representation
$\left(G, V\right)$
is multiplicity-free, that is, the associated representation of 
$G$
on polynomials on
$V$%
, 
$\mathbb{C}[V]$%
,
decomposes without multiplicities (see \cite{Le}).
Assume futhermore that
$\left(G, V\right)$
has one-dimensional quotient, i.e., there exists a non-constant polynomial 
$f$
on
$V$
generating the algebra of
$G'$%
-invariant polynomials
(%
$\mathbb{C}[V]^{G'} = \mathbb{C}[f]$%
) such that
$f\not\in \mathbb{C}[V]^{G}$%
.
As usual
$\mathcal{D}_{\mathbb{P}(V)}$ (resp. $%
\mathcal{D}_{V}$%
) stands for the sheaf of rings of differential operators on the projective space $%
\mathbb{P}(V)$
(resp. 
$V$%
). Let
$U\left(\mathfrak{g}\right)$
be the universal enveloping algebra of the Lie algebra of
$G$%
. If the algebra of
$G$%
-invariant differential operators 
$\Gamma\left(V, \mathcal{D}_{V}\right)^G$
is the image of the center of
$U\left(\mathfrak{g}\right)$ 
under the differential of the
$G$%
-action then the representation 
$\left(G, V\right)$
is said to be of Capelli type.
It is known from V. G. Kac (see \cite{Kac})
that such representations have finitely many
$G$%
-orbits 
(resp.
$\mathbb{P}\left(G\right)$%
-orbits) denoted by
$V_{k}$
(resp.
$\widetilde{V}_k$%
).
Let
$\Lambda
:=\bigcup\limits_{k=0}^{n}T_{V_{k}}^{\ast }V$ 
(resp.
$\widetilde{\Lambda}
:=\bigcup\limits_{k=1}^{n}T^{\ast }_{{\widetilde{V}}_{k}}{\mathbb{P}\left(V\right)}$%
)
be the Lagrangian variety union of conormal bundles to the orbits of the action of
$G$
(resp. 
$\mathbb{P}\left(G\right)$%
)
on 
$V$
(resp.
$\mathbb{P}\left(V\right)$%
).
Let 
\textrm{Mod}$%
_{\Lambda }^{\mathrm{rh}}\left( \mathcal{D}_{V}\right) $ 
(resp. 
\textrm{Mod}$%
_{\widetilde{\Lambda }}^{\mathrm{rh}}\left( \mathcal{D}_{\mathbb{P}\left(
V\right) }\right)$
stands for the category of regular holonomic 
$\mathcal{D}_{V}$%
-modules
(resp.
$\mathcal{D}_{\mathbb{P}\left(V\right) }$%
-modules)
whose characteristic variety
is contained in 
$\Lambda$
(resp.
$\widetilde{\Lambda }$%
). They form abelian categories.\\
In this paper we study
regular holonomic 
$\mathcal{D}_{\mathbb{P}\left(V\right) }$%
-modules on the projective space
$\mathbb{P}\left(V\right)$
that are objects of the category 
\textrm{Mod}$_{%
\widetilde{\Lambda }}^{\mathrm{rh}}\left( \mathcal{D}_{\mathbb{P}\left(
V\right) }\right) $%
. Note that certain authors were interested in the description of similar categories
(see \cite
{BM}, \cite{L-W}, 
\cite{N1}, \cite{Rai}). 
We introduce
$\overline{\mathcal{A}}:=\Gamma\left( V,\mathcal{D}_{V}\right)^{G'}$
the Weyl algebra on 
$V$
of 
$G'$%
-invariant differential global algebraic sections in
$\mathcal{D}_{V}$%
. This algebra is well understood (see \cite{H-U}, \cite{Le}, \cite{Ru2}). It is a polynomial algebra on a canonically defined set of generators called the Capelli operators
(see \cite{H-U}, \cite{Le}, \cite{Ru2}, \cite{U}).
Let 
$\mathcal{A}$
be the quotient algebra
of 
$\overline{\mathcal{A}}$
by 
the ideal of operators vanishing on 
$G'$%
-invariant polynomials.\\
We consider the category 
${\mathrm{Mod}}^{\mathrm{gr}}\left( \mathcal{A}\right)$ 
whose objects are finitely generated graded 
$\mathcal{A}$%
-modules. This is an abelian category.
Denote by
$\mathcal{C}\subset {\mathrm{Mod}}^{\mathrm{gr}}\left( \mathcal{A}\right)$ 
the full subcategory of graded
$\mathcal{A}$%
-modules generated by''homogeneous sections of integral degree''. 
Let
$%
\mathcal{C}_{0}\subset \mathcal{C}$ 
stands for the subcategory consisting
of modules with support at the origin
$\left\{ 0\right\} $%
. The category
$\mathcal{C}_{0}$ 
is a Serre subcategory. Then there exists an abelian category
$\mathcal{C}/\mathcal{C}_{0}$
and an exact functor
$ \mathcal{C} \longrightarrow \mathcal{C}/\mathcal{C}_{0} $
which is essentially surjective and whose kernel is 
$\mathcal{C}_{0}$%
. We set 
$\mathcal{C}^{\prime}:=\mathcal{C}/\mathcal{C}_{0}$%
.\\
In \cite[Theorem 25]{N1}, we defined the functor 
$\Psi: {\mathrm{Mod}}_{\Lambda }^{\mathrm{rh}}\left( 
\mathcal{D}_{V}\right) \longrightarrow {\mathrm{Mod}}^{\mathrm{gr}}\left( 
\mathcal{A}\right) $
and we proved that it is an equivalence of categories with quasi-inverse the functor
$\Phi:  {\mathrm{Mod}}^{\mathrm{gr}}\left( 
\mathcal{A}\right) \longrightarrow 
{\mathrm{Mod}}_{\Lambda }^{\mathrm{rh}}\left( 
\mathcal{D}_{V}\right)$
 (where
$\Phi\left(T\right):= \mathcal{D}_{V}\underset{\mathcal{A}}{\otimes} T$%
).
The main aim of
this paper is to prove that the categories 
\textrm{Mod}$_{\widetilde{%
\Lambda }}^{\mathrm{rh}}\left( \mathcal{D}_{\mathbb{P}\left( V\right)
}\right) $
and the
quotient category
$\mathcal{C}^{\prime}$
are equivalent.
To do so, let
$i: V \backslash\{0\} \xhookrightarrow{} V$
be the natural inclusion and
$\pi :V\backslash
\left\{ 0\right\} \longrightarrow \mathbb{P}\left(V\right) $
be the canonical projection, we need first
to prove that the functor obtained by the composition
$\Psi {i_{+}\pi^{+}}$
takes image in the full subcategory
$\mathcal{C}$
and so there exists a canonical functor
$\Psi^{\prime}$
(the dotted one) such that the following diagram commutes:
$$
\begin{tikzcd}
\mathrm{Mod}_{\widetilde{\Lambda }}^{%
\mathrm{rh}}\left( \mathcal{D}_{\mathbb{P}\left( V\right) }\right) \arrow[r,"{i_+\pi^+}"] \arrow[d,dashed,"{\Psi^{\prime}}"] \arrow[bend right]{dd}[black,swap]{\widetilde{\Psi}} & {\mathrm{Mod}}_{\Lambda }^{\mathrm{rh}}\left( 
\mathcal{D}_{V}\right) \arrow[d,"{\Psi}"]\\
\mathcal{C}\arrow[d,""]\arrow[r,""] & {\mathrm{Mod}}^{\mathrm{gr}}\left( 
\mathcal{A}\right) \\
 \mathcal{C}^{\prime}
\end{tikzcd} 
$$
Then we define the functor
$\widetilde{\Psi}: \mathrm{Mod}_{\widetilde{\Lambda }}^{%
\mathrm{rh}}\left( \mathcal{D}_{\mathbb{P}\left( V\right) }\right) \longrightarrow  \mathcal{C}^{\prime}$
obtained by composition of
$\Psi^{\prime}$
with the functor
$ \mathcal{C} \longrightarrow  \mathcal{C}^{\prime}$%
.\\
Next, by the
composition,  we define another functor
$\pi_+i^+\Phi_{|{\mathcal{C}}}:  \mathcal{C} \longrightarrow  \mathrm{Mod}_{\widetilde{\Lambda }}^{%
\mathrm{rh}}\left( \mathcal{D}_{\mathbb{P}\left( V\right) }\right) $
and then we will prove that this functor induces a functor
$\widetilde{\Phi}: \mathcal{C}^{\prime} \longrightarrow \mathrm{Mod}_{\widetilde{\Lambda }}^{%
\mathrm{rh}}\left( \mathcal{D}_{\mathbb{P}\left( V\right) }\right)$%
. The functors 
$\widetilde{\Psi}$
and
$\widetilde{\Phi}$
are quasi-inverse one of each other.
We prove the following theorem:

\begin{theorem-non}
The functor 
$\widetilde{\Psi}: \mathrm{Mod}_{\widetilde{\Lambda }}^{%
\mathrm{rh}}\left( \mathcal{D}_{\mathbb{P}\left( V\right) }\right)
\longrightarrow \mathcal{C}^{\prime}$
is an equivalence of
categories with quasi-inverse the functor
$\widetilde{\Phi}: \mathcal{C}^{\prime} \longrightarrow \mathrm{Mod}_{\widetilde{\Lambda }}^{%
\mathrm{rh}}\left( \mathcal{D}_{\mathbb{P}\left( V\right) }\right)$%
.
\end{theorem-non}

\section{Preliminaries}

Let
$\left(G, V\right)$
be a multiplicity-free representation, i.e., the associated representation of 
$G$
on 
$\mathbb{C}[V]$
decomposes without multiplicities (see \cite{Le}).
As above
$G'$
is the derived subgroup of
$G$%
. Suppose that
$\left(G, V\right)$
has one-dimensional quotient, i.e., there exists a non-constant polynomial 
$f$
on
$V$
such that 
$\mathbb{C}[V]^{G'} = \mathbb{C}[f]$
and
$f\not\in \mathbb{C}[V]^{G}$%
. 
Let
$\mathfrak{g}$
be the Lie algebra of
$G$%
,
 and 
 $U\left(\mathfrak{g}\right)$
be its universal enveloping algebra. Let
$Z\left(U\left(\mathfrak{g}\right)\right)$
be the center for
$U\left(\mathfrak{g}\right)$%
.
There is a homomorphism of algebras 
\begin{equation}
\tau : U\left(\mathfrak{g}\right) \longrightarrow \Gamma\left(V, \mathcal{D}_V\right)\cdot
\end{equation}
\begin{definition}
A multiplicity-free representation 
$\left(G, V\right)$
is said to be of Capelli type if it is irreducible and
$$\tau\left(Z\left(U\left(\mathfrak{g}\right)\right)\right) = \Gamma\left(V,\mathcal{D}_{V}\right) ^{G}\cdot$$
\end{definition}

\subsection{The polynomial ring of invariant differential operators}
From now on,
$\left(G, V\right)$
is a Capelli type representation.
Let 
$\mathcal{W}$ 
be the Weyl algebra on
$V$%
. We recall the description of 
$\overline{%
\mathcal{A}}:=\Gamma \left( V,\mathcal{D}_{V}\right) ^{G'}\subset \mathcal{W}$
the subalgebra of 
$G'$%
-invariant differential operators. Let
$\theta$
be the Euler vector field on
$V$%
. One knows from Levasseur (see \cite{Le} Lemma 4.10)
that there exist a Bernstein-Sato polynomial 
$b$
and finitely many operators 
$\left(E_j\right)_{j=0,\cdots, r}$
called the Capelli operators.
To each 
$E_j$
is associated a polynomial
$b_{E_j}$%
. Let
$x\in V$%
,
$\dfrac{\partial}{\partial{x}} \in \mathcal{D}_V$%
, we set 
$\Delta := f\left(\frac{\partial}{\partial{x}}\right) $%
,
$\Omega_j := E_j - b_{E_j}\left(\theta\right)$
for
$j=0, \cdots, r$%
. Then by (\cite[Theorem 10, p. 391]{N1}) 
\begin{equation}
\overline{\mathcal{A}} = \mathbb{C}\left\langle f, \Delta, \theta, \Omega_1, \cdots, \Omega_r\right\rangle\cdot
\end{equation}

Consider 
$\mathcal{A}$
as the algebra
$\overline{%
\mathcal{A}}$ 
modulo all those operators vanishing on
$G'$%
-invariant functions. We recall (see
\cite[Corollary 12, p. 393]{N1}) the following proposition:

\begin{proposition}
\label{pu}The algebra 
$\mathcal{A}$ 
is generated over $\mathbb{C}$
by 
$f $%
, 
$\Delta $%
, 
$\theta $ 
such that 
\begin{equation*}
\begin{array}{ccccccc}
\left[ \theta ,f\right] = df \text{,} &  \left[ \theta ,\Delta 
\right] =-d\Delta \text{,} &  f\Delta
= \prod\limits_{j=0}^{d-1}\left( \frac{\theta }{d}+\lambda_j\right) \text{, } &  
\Delta f = \prod\limits_{j=0}^{d-1}\left( \frac{\theta }{d}+\lambda_{j} + 1\right)
\end{array}
\end{equation*}
\text{with}
$d\in \mathbb{N}^{*}, \lambda_0 =0, \;\lambda_j\in \mathbb{Q}\cdot$
\end{proposition}
\subsection{$\mathcal{D}_V$%
-modules on representations of Capelli type}
We refer the reader to \cite{HTT}, \cite{K2}, \cite{K3}, \cite{K-K}
for notions on 
$\mathcal{D}$%
-modules.
Recall that
$\left(G, V\right)$
is a Capelli type representation
and
$\mathbb{P}\left( V\right) $
the projective space of
$V$%
. As in the introduction, there are a finite number of orbits under the action of
$G$
and
$\mathbb{P}\left(G\right)$
respectively.
We denote by 
$V_{k}$ 
(resp. 
$\widetilde{V}_{k}$%
) these orbits.
Let
$\Lambda
:=\bigcup\limits_{k=0}^{n}T_{V_{k}}^{\ast }V$ 
(resp.
$\widetilde{\Lambda }%
:=\bigcup\limits_{k=1}^{n}T_{\widetilde{V}_{k}}^{\ast }\mathbb{P}\left(V\right)$%
) be the union of conormal bundles to these strata. Denote by 
${\mathrm{Mod}}_{\Lambda }^{\mathrm{rh}}\left( \mathcal{D}_{V}\right) $ 
(resp.
${\mathrm Mod}_{\widetilde{\Lambda}}^{\mathrm rh}\left(\mathcal{D}_{\mathbb{P}\left(V\right)}\right)$%
) the category whose objects are regular holonomic 
$\mathcal{D}_{V}$ 
(resp.
$\mathcal{D}_{\mathbb{P}\left(V\right)}$%
)%
-modules\ with characteristic variety contained in
$\Lambda $ 
(resp.
$\widetilde{\Lambda }$%
).
Let
${\mathrm{Mod}}^{\mathrm{gr}}\left( \mathcal{A}\right)$ 
be the category consisting of graded 
$%
\mathcal{A}$%
-modules of finite type.
We have constructed (see \cite[Theorem 25, p. 408]{N1}) the functor 
$\Psi: {\mathrm{Mod}}_{\Lambda }^{\mathrm{rh}}\left( 
\mathcal{D}_{V}\right) \longrightarrow {\mathrm{Mod}}^{\mathrm{gr}}\left( 
\mathcal{A}\right) $
and we proved that
$\Psi$
is an equivalence of categories with quasi-inverse the functor
$\Phi:  {\mathrm{Mod}}^{\mathrm{gr}}\left( 
\mathcal{A}\right) \longrightarrow 
{\mathrm{Mod}}_{\Lambda }^{\mathrm{rh}}\left( 
\mathcal{D}_{V}\right)$
 (where
$\Phi\left(T\right):= \mathcal{D}_{V}\underset{\mathcal{A}}{\otimes} T$%
).
We recall the theorem 25 of \cite{N1}.

\begin{proposition}\label{tu} The functor
$\Psi: {\mathrm{Mod}}_{\Lambda }^{\mathrm{rh}}\left( 
\mathcal{D}_{V}\right) \longrightarrow {\mathrm{Mod}}^{\mathrm{gr}}\left( 
\mathcal{A}\right) $
is an equivalence of categories with quasi-inverse the functor
$\Phi:  {\mathrm{Mod}}^{\mathrm{gr}}\left( 
\mathcal{A}\right) \longrightarrow 
{\mathrm{Mod}}_{\Lambda }^{\mathrm{rh}}\left( 
\mathcal{D}_{V}\right)$%
.
\end{proposition}

\section{$\mathcal{D}_{\mathbb{P}(V)}$%
-modules on the projective space 
$\mathbb{P}\left(V\right)$}

This section deals with the study of regular holonomic 
$\mathcal{D}_{\mathbb{P}\left(V\right) }$%
-modules on the projective space
$\mathbb{P}\left(V\right)$
whose characteristic variety is contained in the Lagrangian variety
$\widetilde{\Lambda }$%
.
These objects form an abelian category we have denoted by
\textrm{Mod}$_{\widetilde{%
\Lambda }}^{\mathrm{rh}}\left( \mathcal{D}_{\mathbb{P}\left( V\right)
}\right) $%
.\\
We recall that 
a section 
$s$ 
of a 
$\mathcal{D}_{V}$%
-module
$%
\mathcal{N}$
is said to be homogeneous of integral degree 
$p\in \mathbb{Z}$%
,
if there exists 
$j$ $\in \mathbb{N}$ 
such that 
$(\theta -p)^{j}s=0$%
.\\
As in the introduction, recall that
$\mathcal{C}\subset {\mathrm{Mod}}^{\mathrm{gr}}\left( 
\mathcal{A}\right)$ 
is the full subcategory consisting of graded 
$\mathcal{A}$%
-modules of finite type generated by homogeneous sections of integral degree. Let
$\mathcal{C}%
_{0}\subset \mathcal{C}$ 
be the Serre subcategory consisting of modules
supported by the origin 
$\left\{ 0\right\} $%
. Then by 
\cite[Lemma 12.10.6]{H}
there exits an abelian category
$\mathcal{C}/\mathcal{C}_{0}$
and an exact functor
$\chi:
\mathcal{C} \longrightarrow \mathcal{C}/\mathcal{C}_{0}$
which is essentially surjective and whose kernel is 
$\mathcal{C}_0$
denoted
$\mathrm{Ker}{\left(\chi\right)} =
\mathcal{C}_0$%
.

\subsection{Construction of a functor
$\widetilde{\Psi}: \mathrm{Mod}_{\widetilde{\Lambda }}^{%
\mathrm{rh}}\left( \mathcal{D}_{\mathbb{P}\left( V\right) }\right) \longrightarrow  \mathcal{C}^{\prime}$
}
 Let
$i: V \backslash\{0\} \xhookrightarrow{} V$
be the open embedding and
$\pi :V\backslash
\left\{0\right\} \longrightarrow \mathbb{P}\left(V\right) $
be the canonical projection.
Recall that the functor
$\Psi: {\mathrm{Mod}}_{\Lambda }^{\mathrm{rh}}\left( 
\mathcal{D}_{V}\right) \longrightarrow {\mathrm{Mod}}^{\mathrm{gr}}\left( 
\mathcal{A}\right) $
is an equivalence of categories
(see proposition \ref{tu}).
We prove the following theorem:

\begin{theorem}\label{t00}
The functor obtained by composition
$\Psi {i_{+}\pi^{+}}$
takes image in the full subcategory
$\mathcal{C}$%
.
Then there exists a canonical functor
${\Psi^{\prime}}: \mathrm{Mod}_{\widetilde{\Lambda }}^{%
\mathrm{rh}}\left( \mathcal{D}_{\mathbb{P}\left( V\right) }\right) \longrightarrow  \mathcal{C}$
such that the following diagram commutes:
$$
\begin{tikzcd}
\mathrm{Mod}_{\widetilde{\Lambda }}^{%
\mathrm{rh}}\left( \mathcal{D}_{\mathbb{P}\left( V\right) }\right) \arrow[r,"{i_+\pi^+}"] \arrow[d,dashed,"{\Psi^{\prime}}"] \arrow[bend right]{dd}[black,swap]{\widetilde{\Psi}} & {\mathrm{Mod}}_{\Lambda }^{\mathrm{rh}}\left( 
\mathcal{D}_{V}\right) \arrow[d,"{\Psi}"]\\
\mathcal{C}\arrow[d,"\chi"]\arrow[r,""] & {\mathrm{Mod}}^{\mathrm{gr}}\left( 
\mathcal{A}\right) \\
 \mathcal{C}^{\prime}
\end{tikzcd} 
$$
\end{theorem}

\begin{proof}
 Let 
$\mathcal{M}$ 
be a regular holonomic 
$\mathcal{D}_{\mathbb{P\left(V\right)}}$%
-module on the projective space
$\mathbb{P}\left(V\right)$
whose characteristic variety is contained in
$\widetilde{\Lambda}$
, i.e., 
$\mathcal{M}$
is an object of the category
$\mathrm{%
Mod}_{\widetilde{\Lambda }}^{\mathrm{rh}}\left( \mathcal{D}_{\mathbb{P}%
\left( V\right) }\right) $%
. By \cite[Corollary 5.4.8.]{K-K}) the inverse image of
$%
\mathcal{M}$ 
by the canonical projection 
$\pi :V\backslash \left\{ 0\right\}
\longrightarrow \mathbb{P}\left( V\right) $
is a regular holonomic $%
\mathcal{D}_{V\backslash \left\{ 0\right\} }$%
-module. Then \cite[Corollary 5.1.11.]{K-K} says that
$\pi ^{+}\left( \mathcal{M}\right) $ 
carries a good filtration 
$\mathcal{M}%
=\bigcup\limits_{k\in \mathbb{Z}}\mathcal{M}_{k}$%
. Since each
$\mathcal{M}_{k}$
is a coherent analytic sheaf on 
$\mathbb{P}\left(V\right)$
then there exists an integer
$n\left(\mathcal{M}_{k}\right)$
such that, for all
$n \geq n\left(\mathcal{M}_{k}\right)$
the 
$\mathcal{O}_{\mathbb{P}\left( V\right)} $%
-module
$\mathcal{M}%
_{k}\underset{\mathcal{O}_{\mathbb{P}\left( V\right) }}{\otimes }\mathcal{O}%
\left( n\right) $
is generated by its global sections 
$H^{0}\left( \mathbb{P}\left( V\right) ,%
\mathcal{M}_{k}\otimes \mathcal{O}\left( n\right) \right)$
(see \cite[Lemme 8]{S}). Put
$\mathcal{D}\left( n\right) :=\mathcal{O}\left( n\right)
\otimes \mathcal{D}_{\mathbb{P}\left( V\right) }\otimes \mathcal{O}\left(
-n\right) $)%
, we see that
$\mathcal{M}%
\otimes \mathcal{O}\left( n\right) $
is a 
$\mathcal{D}\left( n\right) $%
-module. Then the sections of the module
$\pi ^{\ast }\left( \mathcal{M}\otimes 
\mathcal{O}\left( n\right) \right) $
give the homogeneous sections of
integral degree 
$n$
in 
$\pi ^{+}\left( \mathcal{M}\right) $%
. Thus the inverse image
$\pi^{+}\left( \mathcal{M}\right) $ 
is generated over 
$\mathcal{D}%
_{V\backslash \left\{ 0\right\} }$
by finitely many homogeneous global
sections of '' integral degree''
$p\in \mathbb{Z}$%
.\\
Now, taking the direct image by the inclusion
$i: V \backslash\{0\} \xhookrightarrow{} V$
of the
$\mathcal{D}%
_{V\backslash \left\{ 0\right\} }$%
-module
$\pi^{+}\left( \mathcal{M}\right) $%
, we obtain a regular holonomic 
$\mathcal{D}_{V}$%
-module
${i_{+}\pi^{+}}\left( \mathcal{M}\right)$
in the category 
\textrm{Mod}$_{\Lambda }^{\mathrm{rh}}\left( \mathcal{D}%
_{V}\right) $
(see \cite[Theorem 6.2.1]
{K-K}).
According to what preceeds, this last module is generated over
$\mathcal{D}_{V}$
by its global sections homogeneous of integral degree.
By proposition \ref{tu} we know that
the functor 
$\Psi: {\mathrm{Mod}}_{\Lambda }^{\mathrm{rh}}\left( 
\mathcal{D}_{V}\right) \longrightarrow {\mathrm{Mod}}^{\mathrm{gr}}\left( 
\mathcal{A}\right) $
is an equivalence of categories. So, Applying
$\Psi$
on 
${i_{+}\pi^{+}}\left( \mathcal{M}\right)$
leads to a graded 
$\mathcal{A}$%
-module of finite type for the Euler vector field
$\theta$%
, i.e, 
$\Psi\left({i_{+}\pi^{+}}\left( \mathcal{M}\right)\right)$
is an of object of the category
${\mathrm{Mod}}^{\mathrm{gr}}\left( 
\mathcal{A}\right) $
. Since
$\Psi\left({i_{+}\pi^{+}}\left( \mathcal{M}\right)\right)$
is also generated by homogeneous sections of integral degree, then it is an object of the full subcategory
$\mathcal{C}$%
. This leads to a canonical functor
${\Psi^{\prime}}: \mathrm{Mod}_{\widetilde{\Lambda }}^{%
\mathrm{rh}}\left( \mathcal{D}_{\mathbb{P}\left( V\right) }\right) \longrightarrow  \mathcal{C}$
such that the above diagram commutes. Making the composition of the functor
${\Psi^{\prime}}$
with the functor
$\chi: \mathcal{C} \longrightarrow \mathcal{C}^{\prime}$%
,
we obtain the functor
$\widetilde{\Psi}: \mathrm{Mod}_{\widetilde{\Lambda }}^{%
\mathrm{rh}}\left( \mathcal{D}_{\mathbb{P}\left( V\right) }\right) \longrightarrow  \mathcal{C}^{\prime}$%
.
\end{proof}

\subsection{Construction of a functor
${\widetilde\Phi}: \mathcal{C}^{\prime} \longrightarrow
\mathrm{Mod}_{\widetilde{\Lambda }}^{\mathrm{rh}}\left( \mathcal{D}_{\mathbb{%
P}\left(V\right) }\right) $}
Recall (proposition \ref{tu}) that the functor  
$\Phi:  {\mathrm{Mod}}^{\mathrm{gr}}\left( 
\mathcal{A}\right) \longrightarrow 
{\mathrm{Mod}}_{\Lambda }^{\mathrm{rh}}\left( 
\mathcal{D}_{V}\right)$
is the quasi-inverse functor of
$\Psi $%
. Let 
$\Phi_{|\mathcal{C}} $
be the restriction of the functor
$\Phi$
on the full subcategory
$\mathcal{C}$%
. Consider
$T$ 
a graded 
$\mathcal{A}$%
-module in
$\mathcal{C}$%
. The inverse image of
$\Phi\left(T\right)$
by the inclusion
$i$
is the regular holonomic 
$\mathcal{D}_{V\backslash \left\{
0\right\} }$%
-module
$i^{+}\left( \mathcal{D}_{V}\underset{\mathcal{A}}{\otimes }%
{T}\right)
$%
.
The direct image of
this last module
by the canonical projection
$\pi$
is a regular holonomic 
$%
\mathcal{D}_{\mathbb{P}\left(V\right)} $%
-module
on the projective space
$%
{\mathbb{%
P}\left(V\right) } $%
, i.e., an object
of the category $%
\mathrm{Mod}_{\widetilde{\Lambda }}^{\mathrm{rh}}\left( \mathcal{D}_{\mathbb{%
P}\left(V\right) }\right) $%
.
We have defined the functor
$\Phi^\prime := \pi_+i^+\Phi_{|\mathcal{C}}: \mathcal{C} \longrightarrow \mathrm{Mod}_{\widetilde{\Lambda }}^{\mathrm{rh}}\left( \mathcal{D}_{\mathbb{%
P}\left(V\right) }\right)  $
obtained by the composition of the functors
$\Phi_{|\mathcal{C}} $
,
$i^+ $
and
$\pi_+$%
. We prove that this functor induces a functor
${\widetilde\Phi}: \mathcal{C}^{\prime} \longrightarrow
\mathrm{Mod}_{\widetilde{\Lambda }}^{\mathrm{rh}}\left(\mathcal{D}_{\mathbb{P}\left(V\right)}\right)$%
.
\begin{proposition}\label{td}
The functor
$\Phi^\prime :=\pi_+i^+\Phi_{|\mathcal{C}}: \mathcal{C} \longrightarrow{\mathrm{Mod}}_{\Lambda }^{\mathrm{rh}}\left( 
\mathcal{D}_{V}\right) $
induces a functor
${\widetilde\Phi}: \mathcal{C}^{\prime} \longrightarrow
\mathrm{Mod}_{\widetilde{\Lambda }}^{\mathrm{rh}}\left( \mathcal{D}_{\mathbb{%
P}\left(V\right) }\right) $
which is faithful.
\end{proposition}
\begin{proof}
Recall that
$\mathcal{C}%
_{0}\subset \mathcal{C}$ 
(the subcategory of graded modules supported by
$ 0$%
)
is a Serre subcategory.
We have defined an exact functor
$\Phi^{\prime}:=
\pi_+i^+\Phi_{|\mathcal{C}}: \mathcal{C} \longrightarrow \mathrm{Mod}_{\widetilde{\Lambda }}^{\mathrm{rh}}\left( \mathcal{D}_{\mathbb{%
P}\left(V\right) }\right)  $
such that
the subcategory
$\mathcal{C}_0$
is contained in the kernel of the functor
$\Phi^{\prime}$%
, i.e.,
$\mathcal{C}_0 \subset \mathrm{Ker}\left(
\Phi^\prime\right)
$%
. Note that according to
\cite[Lemma 12.10.6]{H}, the quotient 
category
$\mathcal{C}^{\prime}$
and the functor
$\chi: \mathcal{C} \longrightarrow \mathcal{C}^{\prime}$
are characterized by the universal property:
Since
$\Phi^{\prime}$
is such that
$\mathcal{C}_0\subset \mathrm{Ker}\left(
\Phi^{\prime}\right)
$%
,
then there exits a factorization
$\Phi^\prime
= \widetilde{\Phi}\circ \chi$
for a unique exact functor
${\widetilde\Phi}: \mathcal{C}^{\prime} \longrightarrow
\mathrm{Mod}_{\widetilde{\Lambda }}^{\mathrm{rh}}\left( \mathcal{D}_{\mathbb{%
P}\left(V\right) }\right) $%
. Moreover we can see that the Serre subcategory
$\mathcal{C}%
_{0}$ 
is exatly the kernel of the functor
$ \Phi^\prime$%
, i.e.,
$\mathcal{C}_0 = \mathrm{Ker}\left(\Phi^\prime\right)$%
. Then by
\cite[Lemma 12.10.7]{H}
the induced functor
$ \widetilde{\Phi}:\mathcal{C}^\prime \longrightarrow \mathrm{Mod}_{\widetilde{\Lambda }}^{\mathrm{rh}}\left( \mathcal{D}_{\mathbb{%
P}\left(V\right) }\right) $ 
is faithfull.
\end{proof}

\subsection{Equivalence of categories}
The functor
$ \widetilde{\Phi}$
is essentially surjective:
Indeed, let
$\mathcal{M}\in \mathrm{Mod}_{\widetilde{\Lambda }}^{\mathrm{rh}}\left( \mathcal{D}_{\mathbb{%
P}\left(V\right) }\right) $ 
be a regular holonomic
$\mathcal{D}_{\mathbb{%
P}\left(V\right) }$%
-module on the projective space
$\mathbb{%
P}\left(V\right) $
. Define
$T := \widetilde{\Psi}\left(\mathcal{M}\right)$%
, then
$\widetilde{\Phi}\left(T\right) =\left(\widetilde{\Phi}\circ \widetilde{\Psi}\right)\left(\mathcal{M}\right) = 
\left(\widetilde{\Phi}\circ\chi\circ\Psi^\prime\right)
\left(\mathcal{M}\right) = \left(\Phi^\prime\circ\Psi^\prime\right)\left(\mathcal{M}\right) =
\pi_+i^{+}\left(\Phi_{|\mathcal{C}}\circ\Psi\right) {i_{+}\pi^{+}}
\left(\mathcal{M}\right) = \pi_+i^{+}{i_{+}\pi^{+}}
\left(\mathcal{M}\right) 
\simeq \mathcal{M}
$
since
$\Phi\circ{\Psi} = {\rm Id}_{\mathrm{Mod}_{\widetilde{\Lambda }}^{\mathrm{rh}}\left( \mathcal{D}_{\mathbb{%
P}\left(V\right) }\right) }$
(see proposition \ref{tu})%
. As
$\widetilde{\Phi}$
is also faithful
(see proposition \ref{td}), we 
obtain the following theorem:
\begin{theorem}
\label{tt}The
category
$\mathrm{Mod}_{\widetilde{\Lambda }}^{%
\mathrm{rh}}\left( \mathcal{D}_{\mathbb{P}\left( V\right) }\right) $ 
and the
quotient category 
$\mathcal{C}^{\prime}$
are equivalent.

\end{theorem}

\section{Example\label{3.3}: Quiver description of regular holonomic
$\mathcal{D}_{\mathbb{P}(V)}$%
-modules on the projective space of complex skew-symmetric matrices}

Let 
$\left(G, V\right) = \left( GL\left(2m, \mathbb{C}\right), \Lambda^2\left(\mathbb{C}^N\right)\right)$
be the 
$2m\times 2m$%
-skew-symmetric matrices under the action of general linear group. In this case we classify the objects of the quotient category
$\mathcal{C}%
^{\prime}:=\mathcal{C}/\mathcal{C}_{0}$%
.
Let
$X = (x_{ij})$
be a 
$2m\times{2m}$%
-skew-symmetric matrix, 
$\frac{\partial}{\partial{X}} =\left(\frac{\partial}{\partial{x_{ij}}}\right) \in \mathcal{D}_V$%
, we denote by
$f := pf(X)$
be the pfaffian for
$X$%
,
$\Delta := f\left(\frac{\partial}{\partial{X}}\right)$
and
$\theta$
the Euler vector field on
$\Lambda^2\left(\mathbb{C}^N\right)$%
. One knows (see \cite[Corollary 8]{N2}) that the quotient algebra of
$SL\left(2m, \mathbb{C}\right)$%
-invariant differential operators
$\mathcal{A}$
is generated by 
$f$%
,
$\Delta$%
, 
$\theta$
such that
\begin{equation*}
\begin{array}{ccccccc}
\left[ \theta ,f\right] = mf \text{,} & \left[ \theta ,\Delta 
\right] =-m\Delta \text{,} &  f\Delta
=\prod\limits_{j=0}^{m-1}\left( \frac{\theta }{m} + 2j\right) \text{, } & 
\Delta f = \prod\limits_{j=0}^{m-1}\left( \frac{\theta }{m} + 2j + 1\right) \text{%
.}
\end{array}
\end{equation*}
Actually the quotient category
$\mathcal{C}%
^{\prime }:=\mathcal{C}/\mathcal{C}_{0}$ 
is a quiver category of linear spaces and maps satisfying commutation relations.
So the objects of
$\mathcal{C}^{^{\prime}}$
can be encoded by means of finite diagrams of linear maps. Indeed, a graded
$%
\mathcal{A}$%
-module 
$T$ 
in 
$\mathcal{C}^{^{\prime }}$
defines an infinite diagram consisting of finite dimensional complex vector spaces
$T_{p}$
(with 
$\left( \theta - p\right) $
being nilpotent on each
$T_{p}$%
, 
$p\in \mathbb{Z}$%
) and linear maps between them deduced from the action of 
$\theta $%
, 
$f$%
,
$\Delta $%
: 
\begin{equation}
\cdots \rightleftarrows T_{p}\overset{f}{\underset{\Delta }{%
\rightleftarrows }}T_{p+m}\rightleftarrows \cdots  \label{R0}
\end{equation}
satisfying the above relations and the following
$\left(
\theta - p\right) T_{p}\subset T_{p}$, 
\begin{equation}
 f\Delta = \prod\limits_{j=0}^{m-1}\left( \frac{\theta }{m} + 2j\right), \quad
 \Delta{f} = \prod\limits_{j=0}^{m-1}\left( \frac{\theta }{m} + 2j + 1\right)\text{ on }%
T_{p}\text{.}  \label{R}
\end{equation}

\noindent Any object of the quotient category
$\mathcal{C}^{^{\prime
\prime }}$
is a diagram 
$\widetilde{T}=T$ \textbf{modulo} $\mathcal{C}_{0}$ ($%
T\in $
$\mathcal{C}^{^{\prime }}$) 
\begin{equation}
\cdots \rightleftarrows T_{p}\overset{f }{\underset{\Delta }{%
\rightleftarrows }}T_{p+m}\rightleftarrows \cdots \text{\textbf{modulo} }%
\mathcal{C}_{0}\text{, \ \ }p\in \mathbb{Z}\text{,}
\end{equation}
satisfying the previous relations (\ref{R}). 
$\widetilde{T}$
is described by a finite subset of objects and arrows. From
\cite[section 6, page 130]{N2} we deduce the following:\newline

\noindent $\left( a\right) $\noindent\ If 
$p\equiv 0 \mod m$, then $\widetilde{T}$
is encoded by a diagram of $2m$ elements 
\begin{equation}
\begin{array}{ccccccccc}
T_{-2m(m-1)} & \overset{f }{\underset{\Delta }{\rightleftarrows }} & 
T_{-2m\left( m-2\right) } & \overset{f}{\underset{\Delta }{%
\rightleftarrows }} & \cdots & T_{-m} & \overset{%
f}{\underset{\Delta }{\rightleftarrows }} & T_{0}
\end{array}
\text{\textbf{modulo} }\mathcal{C}_{0}\text{.}
\end{equation}
In the other degrees 
$f $ 
or 
$\Delta $ 
are bijective since
$%
T_{0}\simeq f^{k}T_{0}\simeq T_{mk}$ 
and 
$T_{-(2m-1)m}\simeq \Delta
^{k}T_{-(2m-1)m}\simeq T_{-(2m-1 +k)m}$ ($k\in \mathbb{N}$) 
thanks to the
relations (\ref{R}).\newline

\noindent $\left( b\right) $ If
$p$
is an integer 
(
$p\not\equiv 0 \mod m $)
, then 
$f $ 
and 
$\Delta $ 
are bijective. So the diagram
$\widetilde{T%
}$ 
is defined by one element 
$T_{p}$ 
\textbf{modulo} $\mathcal{C}_{0}$ 
with the nilpotent action of $%
\left( \theta -p\right) $.\newline

\noindent\textbf{Acknowledgements.} The author is grateful to the Max-Planck Institute for Mathematics in Bonn for its hospitality and financial support.

\end{document}